\documentclass[]{amcjoucc}

\RequirePackage{amssymb}

\newcommand{\Aut}{{\rm{Aut}}}

\def\ov{\overline}

\def\Ga{\Gamma}

\def\la{\langle}
\def\ra{\rangle}

\newcommand{\OG}{\mathcal{OG}}

\begin{document}

\begin{frontmatter}   

\titledata{A normal quotient analysis for some families of oriented four-valent graphs}           
{}                 







\begin{center}
\authordatanoaffil{Jehan A. Al-bar}
{jalbar@kau.edu.sa; jaal[underscore]bar@hotmail.com}
{This project was funded by the Deanship of Scientific Research (DSR), King Abdulaziz University, Jeddah, 
under grant no. HiCi/H1433/363-1. The authors, therefore, acknowledge with thanks DSR technical and financial support.},
\authordatanoaffil{Ahmad N. Al-kenani}
{analkenani@kau.edu.sa; aalkenani10@hotmail.com}
{},
\authordatanoaffil{Najat Mohammad Muthana}
{nmuthana@kau.edu.sa; najat[underscore]muthana@hotmail.com}
{}
\affiliation{King Abdulaziz University, Jeddah, Saudi Arabia}
\end{center}

\authordata{Cheryl E. Praeger}            
{The University of Western Australia, Crawley, Australia; also affiliated with King Abdulaziz University, Jeddah, Saudi Arabia}
{cheryl.praeger@uwa.edu.au}
{corresponding author}


\begin{abstract}
We analyse the normal quotient structure of several infinite families of 
finite connected edge-transitive, four-valent oriented graphs. These 
families were singled out by Maru\v{s}i\v{c} and others to illustrate 
various different internal structures for these graphs in terms of their 
alternating cycles (cycles in which consecutive edges have opposite orientations). 
Studying the normal quotients gives fresh insights into these oriented graphs: 
in particular we discovered some unexpected `cross-overs' between these graph 
families when we formed normal quotients. We determine which of these 
oriented graphs are `basic', in the sense that their only proper normal 
quotients are degenerate.  Moreover, we show that the three types of edge-orientations 
studied are the only orientations, of the underlying undirected graphs in 
these families, which are invariant under a group action which is both vertex-transitive and edge-transitive.
\end{abstract}

\end{frontmatter}   


\section{Introduction}\label{sec:intro} 

The graphs we study are simple, connected, undirected graphs 
of valency four, admitting an orientation of their edges preserved by a vertex-transitive and 
edge-transitive subgroup of the automorphism group, that is, graphs 
of valency four admitting a \emph{half-arc-transitive} group action. Many of these graphs arise as medial graphs for
regular maps on Riemann surfaces \cite{Mar1, MN}. The work of Maru\v{s}i\v{c}, 
summarised in \cite{Mar1}, demonstrated the importance of a certain family of cyclic subgraphs for 
understanding the internal structure of these graphs, namely their \emph{alternating cycles}. These are cycles in which
each two consecutive edges have opposite orientations. They were introduced in \cite{Mar2} and 
their basic properties were studied in \cite{Mar2, MP}.  

The families of oriented graphs studied in 
this paper were singled out in \cite{MN, MP} because they demonstrate three different extremes 
for the structure of their alternating cycles: namely the alternating cycles are `loosely attached', 
`antipodally attached' or `tightly attached' (see Subsection~\ref{sec:alt}).
These families are based on one of two infinite families of underlying 
unoriented valency four graphs (called $X(r)$ and $Y(r)$ for positive integers $r$), 
and for each family three different edge-orientations are induced by three different 
subgroups of their automorphism groups (Section~\ref{sub:results}); the three different edge-orientations
correspond to the three different attachment properties of their alternating cycles. 
In this paper it is shown in Theorem~\ref{thm:allorients} that these are essentially the only 
edge-orientations of the underlying graphs which are invariant under a half-arc-transitive group action.
 
Our approach is to study the normal quotients of these `$X$-graphs' and `$Y$-graphs' 
(as explained below), and in doing so we discover that graphs from a third family 
arise,  the `$Z$-graphs'. We find (Theorem~\ref{thm:main})  
that normal quotients of the $X$-graphs are sometimes 
$X$-graphs, sometimes $Y$-graphs, and sometimes $Z$-graphs, and the same is true for 
normal quotients of the $Y$-graphs. 
In addition we determine in Theorem~\ref{thm:mainB} those oriented graphs in these families 
which are `basic' in the sense that all their proper normal quotients are degenerate (see Subsection~\ref{sec:nquot}). 

\subsection{Graph-group pairs and their normal quotients}\label{sec:nquot}
The normal quotient approach was introduced in \cite{janc1} to study oriented graphs,
focusing on the structure of certain quotient graphs rather than subgraphs. For a connected oriented four-valent graph
$\Ga$ with corresponding vertex- and edge-transitive group $G$ preserving the edge-orientation, a \emph{normal quotient} of $(\Ga,G)$ is
determined by a normal subgroup $N$ of $G$. It is the graph $\Ga_N$ with vertices the $N$-orbits on 
the vertices of $\Ga$, and with distinct $N$-orbits $B, C$ adjacent provided there is some edge of 
$\Ga$ between a vertex of $B$ and a vertex of $C$. The normal quotient theory in \cite{janc1} asserts
(with specified degenerate exceptions) that the normal quotient $\Ga_N$ has valency four, and inherits an edge-orientation from $\Ga$
which is preserved by the quotient group $G/N$ acting transitively on 
vertices and edges, and, moreover, $\Ga$ is a cover of $\Ga_N$ (that is, for 
adjacent $N$-orbits $B, C$, each vertex of $B$ is adjacent to exactly one vertex of $C$). 

We call a pair $(\Ga,G)$ \emph{basic} if the only proper normal quotients (that is, taking $N\ne 1$) are 
degenerate. Each pair $(\Ga,G)$ has at least one basic normal quotient, (more details are given in Section~\ref{sec:prelim},
and see \cite{janc1, janc2}).
For the pairs $(\Ga,G)$ we study in this paper, there is always a basic normal quotient which lies in one of 
the families (possibly a different family from the family containing $(\Ga,G)$).
Thus, although all the graph-group pairs in a given family share the same properties 
of their alternating cycles (Remark~\ref{rem:xy}), the structure of the family
can be further elucidated by studying the much smaller subfamily of basic pairs.

\subsection{The alternating cycles of {Maru\v{s}i\v{c}}} \label{sec:alt}

Let $\OG(4)$ denote the set of all graph-group pairs $(\Ga,G)$, where $\Ga$ is a connected graph 
of valency four, $G$ is a vertex-transitive and 
edge-transitive subgroup of automorphisms, and $G$ preserves an orientation of the edges.
In particular $G$ is not transitive on the arcs of $\Ga$ (ordered pairs of vertices which form an edge),
and such a group action is called \emph{half-arc-transitive}. 

Alternatively we could interpret each pair $(\Ga,G)\in\OG(4)$ as consisting of a connected, 
undirected graph $\Ga$ of valency four, and a group $G$ of automorphisms acting half-arc-transitively. 
Such a group action determines an edge-orientation of $\Ga$ (up to reversing the orientation on every edge).
It is possible for a single graph to have different edge-orientations determined by different subgroups of 
automorphisms. This is the case with the examples in Subsection~\ref{sub:results}~(a) and (b). 
By viewing the fundamental objects of study as undirected graphs endowed with (perhaps several) edge-orientations, 
we are able to give a unified discussion of all possible edge-orientations preserved 
by half-arc-transitive group actions.  
Our notation is therefore slightly different from the papers \cite{MN, MP} where different names 
for the same graph are used for different edge-orientations. 

The alternating cycles of $(\Ga,G)$ (defined above), are determined by the edge-orientation. 
They partition the edge set, they 
all have the same even length, denoted by $2\cdot r(\Ga,G)$, where $r(\Ga,G)$ is called the \emph{radius} of $(\Ga,G)$, 
and if two alternating cycles share at least one vertex, then they 
intersect in a constant number $a(\Ga,G)$ of vertices, called the \emph{attachment number}, such that $a(\Ga,G)$ 
divides $2\cdot r(\Ga,G)$ \ ~\cite[Proposition 2.4]{Mar2}. 
The radius can be any integer greater than 1 (\cite[Section 3]{Mar2} and \cite[Section 4]{MP}),
and the attachment number can be any positive integer \cite[Theorem 1.5]{MW}. It is possible that $a(\Ga,G)=2\cdot r(\Ga,G)$
and all such pairs $(\Ga,G)$ were characterised by Maru\v{s}i\v{c} in \cite[Proposition 2.4(ii)]{Mar2}. 

Otherwise, if $a(\Ga,G)<2\cdot r(\Ga,G)$,
then $\Ga$ is a cover of a (possibly smaller) quotient graph $\Ga'$ admitting a (possibly unfaithful) action $G'$ of $G$
such that $(\Ga',G')\in\OG(4)$ and the attachment number $a(\Ga',G')$ is either $1, 2$ 
or $r(\Ga',G')$, \cite[Theorem 1.1  and Theorem 3.6]{MP}. 
For this reason attention has focused on families of examples $(\Ga,G)\in\OG(4)$ 
for which $a(\Ga,G)$ is $1, 2$ or $r(\Ga,G)$, and such pairs are said to be
\emph{loosely attached, antipodally attached} or \emph{tightly attached}, respectively.
As we already mentioned, graphs in the families we examine have one of these properties, 
and the structure of their alternating cycles was studied by Maru\v{s}i\v{c} 
and others \cite{CPS, HKM, Mar2, MN, MP, MS, MW, PSV}.
  
%


\subsection{The families of oriented graphs and our results}\label{sub:results} 
We describe the underlying graphs and their edge-orientations.

\smallskip\noindent
(a)  The first three families of oriented graphs are all based on the 
same family of graphs, namely the Cartesian product $X(r)= C_{2r}\,\square\, C_{2r}$  
of two cycles of length $2r$, for a positive integer $r$.
The graph $X(r)$ has vertex set $\mathbb{Z}_{2r}\times \mathbb{Z}_{2r}$, 
such that $(i,j)$ is adjacent to $(i\pm 1, j)$ and $(i, j\pm 1)$ for
all $i,j \in \mathbb{Z}_{2r}$. If $r=1$ then $X(1)=C_4$, a $4$-cycle. 
If $r\geq2$, then $X(r)$ has valency $4$ and its automorphism group is 
$G(r) = \Aut\,X(r) = D_{4r}\wr Z_2 = D_{4r}^2\rtimes Z_2$, with generators
\begin{align*}
\mu_1    &: (i,j)\longrightarrow (i+1,j), &\hspace{1cm} &\mu_2   : (i,j)\longrightarrow (i,j+1),\\
\sigma _1&: (i,j)\longrightarrow (-i,j),  &\hspace{1cm} &\sigma _2: (i,j)\longrightarrow (i,-j), \\
\tau     &: (i,j)\longrightarrow (j,i).
\end{align*}
The group $G(r)$ is arc-transitive of order $32r^2$, and the stabiliser of the vertex $x=(0,0)$ is 
$G(r)_x=\la \sigma_1,\sigma_2,\tau\ra$, a dihedral group of order 8. In 1999, Maru\v{s}i\v{c} and 
the fourth author \cite{MP} defined three edge-orientations on the graphs $X(r)$, each of which 
corresponds to a half-arc-transitive action of a certain subgroup of $G(r)$, as follows. 
The second edge-orientation was also studied by Maru\v{s}i\v{c} and Nedela \cite{MN}. 
As mentioned above, by Theorem~\ref{thm:allorients}, these are the only edge-orientations
preserved by a half-arc-transitive subgroup of $\Aut\,X(r)$,
up to conjugation in the automorphism group, and reversing the orientations on all edges. We note 
that elements of $\mathbb{Z}_{2r}$ have a well-defined parity (even or odd).

\begin{itemize}
\item[(a.1)] Define the first edge-orientation by
$$(i,j)\rightarrow (i,j+1) \hspace{.2cm} \mbox{if $i$ is even,}  \hspace{.2cm} (i,j)\leftarrow (i,j+1) \hspace{.2cm} \mbox{if $i$ is odd}$$ 
$$(i,j)\rightarrow (i+1,j) \hspace{.2cm} \mbox{if $j$ is even,} \hspace{.2cm} (i,j)\leftarrow (i+1,j)\hspace{.2cm} \mbox{if $j$ is odd}$$
and the corresponding group 
$$G_1(r):=\la \mu_1\sigma_2, \mu_2\sigma_1, \tau\ra.$$ 
Note that $G_1(r)_x=\la\tau\ra\cong Z_2$, and $|G_1(r)|=8r^2$.
\item[(a.2)] Define the second edge-orientation by
$$(i,j)\rightarrow (i, j+1), \hspace{.3cm} (i,j)\leftarrow (i+1,j) \hspace{.3cm}\mbox{if}\hspace{.2cm} i+j\hspace{.2cm}\mbox{is even}$$
$$(i,j)\leftarrow (i,j+1), \hspace{.3cm} (i,j)\rightarrow (i+1,j)\hspace{.3cm}\mbox{if}\hspace{.2cm} i+j\hspace{.2cm}\mbox{is odd}$$
and the corresponding group 
$$G_2(r):=\la \mu_1\mu_2, {\mu_1}^2,  \sigma _1, \sigma _2, \tau\mu_1\ra.$$ 
Note that $G_2(r)_x=\la\sigma_1, \sigma_2\ra\cong Z_2^2$, and $|G_2(r)|=16r^2$.
\item[(a.3)] Define the third edge-orientation by
$$(i,j)\rightarrow (i,j+1), \hspace{.3cm} (i,j)\rightarrow (i+1,j)\hspace{.3cm}\mbox{for all} \hspace{.2cm}i\hspace{.2cm} \mbox{and}\hspace{.2cm} j$$
and the corresponding group 
$$G_3(r):=\la \mu_1, \mu_2, \tau\ra.$$ 
Note that $G_3(r)_x=\la\tau\ra\cong Z_2$, and $|G_3(r)|=8r^2$.
\end{itemize}
In Remark~\ref{rem:xir} we discuss briefly how these three edge-orientations may be visualised. In neither of the papers \cite{MN, MP} where these graphs were previously studied are the generators of the groups $G_k(r)$ defined explicitly, and we need this information for our analysis. Indeed in order to analyse these families of oriented graphs we need an additional graph family related to the third edge orientation.
\begin{itemize}
\item[(a.4)] Let $s$ be an \emph{odd integer}, $s\geq3$, and let $Z(s)= C_{s}\,\square\, C_{s}$ (that is, the graph $X(r)$ with $2r$ replaced by  
$s$); let $\mu_1,\mu_2,\tau$ have the same meanings as above (as permutations of the set $\mathbb{Z}_s\times\mathbb{Z}_s$); 
define the edge-orientation as in (a.3); and call the corresponding group $G_{3Z}(s)=\la\mu_1,\mu_2,\tau\ra$.
\end{itemize}

\smallskip\noindent
(b)  The last three families of oriented graphs are all based on certain quotients of the graphs 
in part (a). Define the (undirected) graph $Y(r)$ as the quotient of $X(r)$ modulo the orbits of 
the normal subgroup $M(r):=\la(\mu_1\mu_2)^r\ra$ of $G(r)$. Thus the vertices of $Y(r)$ are the 
$2$-element subsets $\{(i,j),(i+r,j+r)\}$, for $i, j\in\mathbb{Z}_{2r}$, and the vertex 
$\{(i,j),(i+r,j+r)\}$ of $Y(r)$ is adjacent to $\{(i\pm1,j),(i+r\pm1,j+r)\}$ and 
$\{(i,j\pm1),(i+r,j+r\pm1)\}$. The graph $Y(1)=K_2$ and, for $r\geq2$, $Y(r)$ has valency 
$4$, and admits an arc-transitive action of the quotient $H(r):=G(r)/M(r)$. Moreover, $X(r)$ 
is a cover of $Y(r)$. In particular, $Y(2)=K_{4,4}$, a complete bipartite graph.

For $k=2, 3$ with $r\geq2$, and also for $k=1$ with $r$ even, we have $M(r)\leq G_k(r)$ and we 
define $H_k(r):=G_k(r)/M(r)$. Then the graph-group pair 
$(Y(r),H_k(r))\in\OG(4)$ (Lemma~\ref{lem:x2}). It is the  normal quotient of the pair 
$(X(r), G_k(r))$ relative to $M(r)$, and by \cite[Theorem 1.1]{janc1}, $Y(r)$ inherits 
the $k^{th}$-edge-orientation from $X(r)$.  
 
\smallskip\noindent
For any of the graphs $X(r), Y(r), Z(s)$, each of the edge-orientations defined above
is invariant under some edge-transitive subgroup of automorphisms (by Theorem~\ref{tbl:mainB}),
and these are essentially the only such edge-orientations for these graphs.

\begin{theorem}\label{thm:allorients}
Let $\Gamma = X(r)$ with $r\geq2$, or $\Gamma=Y(r)$ with $r\not\in\{1,2,4\}$, 
or $\Gamma=Z(s)$ with $s$ odd, $s > 1$. 
Then, up to conjugation in $\Aut\,\Ga$, and up to reversing  the orientation on each edge, the only edge-orientations of $\Gamma$ 
invariant under a half-arc-transitive subgroup of $\Aut\,\Gamma$ are those defined in Subsection~$\ref{sub:results}$.     
\end{theorem}

\smallskip\noindent
We prove Theorem~\ref{thm:allorients} in Section~\ref{sub:allorients}.
By \cite[Lemma 3.3]{janc1}, each graph-group pair $(\Ga,G)\in\OG(4)$ is a normal cover of at least one basic pair in $\OG(4)$. 
It turns out that the graph-group pairs  $(X(r), G_k(r))$, $(Y(r),H_k(r))$, and $(Z(s), G_{3Z}(s))$ all lie in $\OG(4)$, and are all
normal covers of at least one basic pair from one of these families, but not necessarily from the same family (Remark~\ref{rem:xir}~(c)).
Our main results identify which of these pairs is basic (Theorem~\ref{thm:mainB}), and present some of
their interesting normal quotients (Theorem~\ref{thm:main}). The types of basic graph-group pairs are defined
according to the kinds of degenerate normal quotients they have. These are summarised in Table~\ref{tbl:basictypes}, 
taken from \cite[Table 2]{janc1}.

\begin{table}
\begin{center}
\begin{tabular}{lll}
\hline
Basic Type & Possible $\Ga_N$ for $1\ne N\vartriangleleft G$& Conditions on $G$-action \\ &&on vertices \\ \hline
quasiprimitive & $K_1$ only & quasiprimitive \\
biquasiprimitive & $K_1$ and $K_2$ only ($\Ga$ bipartite) & biquasiprimitive \\
cycle & at least one $C_m$ ($m\geq3$) & at least one quotient action\\ && $D_{2m}$ or $Z_m$ \\ \hline
\end{tabular}
\caption[Basic]{Types of basic pairs $(\Ga,G)$ in $\OG(4)$}
\label{tbl:basictypes}
\end{center}
\end{table}

\begin{theorem}\label{thm:mainB}
Let $(\Ga,G)$ be one of the graph-group pairs in Table~$\ref{tbl:mainB}$. 
Then $(\Ga,G)\in\OG(4)$. Moreover $(\Ga,G)$ is basic if and only if
the conditions in the `Conditions to be Basic' column hold, and in this case, the basic type
is given in the `Basic Type' column. 
\end{theorem}

\begin{table}
\begin{center}
\begin{tabular}{clll}
$(\Ga,G)$         &Conditions          &Conditions to be Basic&Basic Type \\ \hline
$(X(r),G_k(r))$   &all $k$ and $r$                  & $k=1$, $r$ odd prime &cycle  \\ 
$(Y(r),H_k(r))$   &all $k$ and $r$ even&all $k$ and $r=2$     &cycle  \\ 
$(Y(r),H_k(r))$   &$k>1$ and $r$ odd   &$k=2$, $r$ odd prime  &biquasiprimitive            \\ 
$(Z(s),G_{3Z}(s))$&$s\geq3$ odd        &$s$ odd prime         &cycle  \\ 
\hline  
\end{tabular}
\end{center}
\caption{Conditions for $(\Ga,G)$ to be basic in Theorem~\ref{thm:mainB}. (Refer to Lemma~\ref{lem:basic}.)} \label{tbl:mainB}
\end{table}

We verify  membership of $\OG(4)$ in Lemma~\ref{lem:x2}, and prove the assertions in Table~\ref{tbl:mainB} 
in Lemma~\ref{lem:basic}. 
The normal quotients we explore are those modulo the following subgroups of $G(r)$ and/or $G_{3Z}(s)$. We note that $|\mu_i|$ is $2r$ or $s$,
when interpreting $\mu_i$ as an element of $G(r)$ or $G_{3Z}(s)$, respectively.  
For each divisor $a$ of $2r$ or of $s$, as appropriate, define
\begin{align}\label{na} 
N(a)  &= \la \mu_1^{a}, \mu_2^{a}\ra\                            &\mbox{if}\ && a &\mid |\mu_1| \\
M(a)  &= \la (\mu_1\mu_2)^{a}, \mu_1^{2a}\ra                    &\mbox{if}\ && 2a &\mid |\mu_1| \nonumber \\
N(2,+)&= \la \mu_1^2, \mu_2^2, \sigma_1,\sigma_2\ra \leq G_2(r). &           &&&  \nonumber 
\end{align}
If $|\mu_1|=2r$, then each of $N(a), M(a)$ is normal in the full automorphism group $G(r)$ of $X(r)$, and if $a\mid r$, then
$N(2a)$ is a subgroup of $M(a)$ of index $2$; if $|\mu_1|=s$ is odd, then $N(a)\trianglelefteq G_{3Z}(s)$ (and $M(a)$ is not defined). We also consider:
\begin{align}\label{ld} 
J&=\la \mu_1\mu_2\ra         \ \cong {Z}_{t}, &\mbox{and}\quad  K &=\la \mu_1\mu_2^{-1},\tau\ra \cong {D}_{2t}
\end{align}
where $t=|\mu_1|\in\{2r,s\}$, and if $|\mu_1|=2r$, also
\begin{align}\label{ld+} 
J(+)&=\la \mu_1\mu_2, \mu_1^r\ra     &\mbox{and}\quad K(+)&=\la \mu_1\mu_2^{-1},\tau, \mu_1^r\ra. 
\end{align}


\noindent
If $t=2r$ then the four subgroups in \eqref{ld} and \eqref{ld+}  all contain $M(r)$ and are normal in $G_3(r)$, while if $t=s$
then $J, K$ are normal in $G_{3Z}(s)$.  
For an arbitrary subgroup $L$ of $G(r)$, we write 
$\ov L := LM(r)/M(r)$, so, for example, $H(r)=\ov{G(r)}$ and 
$H_k(r)=\ov{G_k(r)}$, and we also consider the subgroups $\ov{M(a)}$ and $\ov{N(a)}$. Note that 
\begin{equation}\label{nm}
M(r)\not\leq N(a)\quad \mbox{if and only if}\quad  \mbox{$\frac{2r}{a}$ is odd}\quad \mbox{if and only if}\quad  \mbox{$\ov{N(a)}=\ov{M(\frac{a}{2})}$}.    
\end{equation}

\begin{theorem}\label{thm:main} 
Let $(\Ga,G)$ be a graph-group pair 
$(X(r), G_k(r))$, $(Y(r),H_k(r))$,  or \\
$(Z(s), G_{3Z}(s))$, where 
$k\in\{1,2,3\}$,  $s\geq3$ is odd, $r\geq2$, and in the case of  $(Y(r),H_1(r))$, $r$ is even. 
\begin{enumerate}
\item[(a)] Then  $(\Ga,G)$ has proper non-degenerate normal quotients $(\Ga_N, G/N)$, for $(\Ga, G), N$ and the `Conditions' as in 
one of the lines of Table~$\ref{tbl:main}$. 
\item[(b)] also  $(\Ga,G)$ has degenerate normal quotients $(\Ga_N, G/N)$, for $N, G$ as in 
one of the lines of Table~$\ref{tbl:main2}$.  
\end{enumerate}
\end{theorem}

We prove parts (a) and (b) of Theorem~\ref{thm:main} in Lemmas~\ref{lem:lines1-4} and \ref{lem:tbl2}, respectively.
We do not claim that Theorem~\ref{thm:main} classifies all the
normal quotients for these graph-group pairs.

\begin{table}
\begin{center}
{\openup 1\jot
\begin{tabular}{clll}
$(\Ga,G)$      &$N$   &$(\Ga_N,G/N)$                &Conditions on $k, a, r$  \\ \hline
$(X(r),G_k(r))$&$N(2a)$&$(X(a),G_k(a))$             &$a<r$   \\ 
$(X(r),G_k(r))$&$M(a)$&$(Y(a),H_k(a))$              &$\frac{2r}{a}$ even, and $a$ even if $k=1$   \\ 
$(X(r),G_3(r))$&$N(a)$&$(Z(a),G_{3Z}(a))$           &$a$ odd   \\ 
$(Y(r),H_k(r))$&$\ov{N(2a)}$&($X(a),G_k(a))$        &$\frac{r}{a}$ even    \\ 
$(Y(r),H_k(r))$&$\ov{M(a)}$&$(Y(a),H_k(a))$         &$\frac{2r}{a}>2$ even, and  $a$ even if $k=1$  \\ 
$(Y(r),H_3(r))$&$\ov{N(a)}$&$(Z(a),G_{3Z}(a))$      &$a$ odd   \\ 
$(Z(s),G_{3Z}(s))$&$N(a)$  &$(Z(a),G_{3Z}(a))$      &$a<s$    \\ 
\hline  
\end{tabular}
}
\end{center}
\caption{Non-degenerate normal quotients of $(\Ga,G)$ for Theorem~\ref{thm:main}, with $N$ as in \eqref{na} and $a>1$.
(Refer to Lemma~\ref{lem:lines1-4}.)} \label{tbl:main}
\end{table}

\begin{table}
\begin{center}
{\openup 1\jot
\begin{tabular}{cccll}
$N$ in $G_k(r)$ &$N$ in $H_k(r)$ &$N$ in $G_{3Z}(s)$ &$(\Ga_N,G/N)$ &Conditions            \\ \hline 
$N(2)$          &$\ov{N(2)}$     & --                &$(C_4,D_8)$   &$k=1,3$,             \\ 
                &                &                   &              &with $r$ even for $H_k(r)$ \\ 
$N(2,+)$        &$\ov{N(2,+)}$   &--                 &$(C_4,Z_4)$   &$k=2$                 \\ 
                &                &                   &              &with $r$ even for $H_k(r)$ \\ 
--              &$\ov{N(2,+)}$   &--                 &$(K_2,Z_2)$   &$k=2$, $r$ odd        \\ 
--              &$\ov{N(2)}$     &--                 &$(K_2,Z_2)$   &$k=3$, $r$ odd        \\ 
$J$             &$\ov{J}$        &$J$                &$(C_{t},D_{2t})$&$k=3$, $t\in\{2r,s\}$   \\ 
$K$             &$\ov{K}$        &$K$                &$(C_{t},Z_{t})$ &$k=3$, $t\in\{2r,s\}$   \\ 
$J(+)$          &$\ov{J(+)}$     &--                 &$(C_{r},D_{2r})$&$k=3$                   \\ 
$K(+)$          &$\ov{K(+)}$     &--                 &$(C_{r},Z_{r})$ &$k=3$                   \\ 
\hline  
\end{tabular}
}
\end{center}
\caption{Degenerate normal quotients of $(\Ga,G)$ for Theorem~\ref{thm:main}, with $N$ as in \eqref{na}, \eqref{ld} or \eqref{ld+}.
(Refer to Lemma~\ref{lem:tbl2}.)} \label{tbl:main2}
\end{table}

\begin{remark}\label{rem:xy}{\rm 
(a) When defining $(Z(s), G_{3Z}(s))$ with $s$ odd, we only consider the third edge-orientation since, for $s$ odd, 
elements of $\mathbb{Z}_s$ do not have a well-defined parity, and so the definitions of the 
first and second edge-orientations do not make sense for the graph $Z(s)$.

Also, when defining $(Y(r), H_k(r))$, we require that $r$ should be even when $k=1$, since when $k=1$ and $r$ is odd, 
there are oriented edges in both directions between adjacent $Y(r)$-vertices, 
so the first edge-orientation of $X(r)$ is not inherited by $Y(r)$.

\smallskip\noindent
(b) The nomenclature in (a.4) may seem clumsy. However we decided to keep the same names for the graphs $X(r), Y(r)$ 
in order to facilitate reference to \cite{MN, MP} for their other properties, as follows.
\begin{enumerate}
 \item[(i)]  It is pointed out in \cite[Section 2]{MP} and \cite[Page 161]{MN}, that for $r\geq3$, the oriented graph-group  
pairs $(X(r), G_k(r))$ are loosely attached with radius $2$, antipodally attached with radius $r$, and tightly 
attached with radius $2r$,  for $k=1, 2, 3$, respectively. Note that the graph $X(r)$ with the $k^{th}$ edge-orientation is called $X_k(r)$ in \cite{MP}. 
\item[(ii)] It is remarked in \cite[Section 2]{MP}, that if $r\geq3$, then the oriented graph-group  pairs $(Y(r), H_k(r))$ 
(with $r$ even if $k=1$) are loosely attached with radius $2$, antipodally attached with radius $r$, and tightly 
attached with radius $r$,  for $k=1, 2, 3$, respectively.  
\item[(iii)] The pairs $(X(r), G_2(r))$ and  $(Y(r), H_2(r))$ were also studied by Maru\v{s}i\v{c} and Nedela, 
where they were characterised in \cite[Props. 3.3, 3.4]{MN} as the only pairs with stabilisers of order at least 4, 
and such that every edge lies in precisely two oriented 4-cycles. 
\end{enumerate}

\smallskip\noindent
(c) By Theorems~\ref{thm:mainB} and~\ref{thm:main}, the graph-group pairs all have basic normal quotients, sometimes more than one. We summarise our findings.
\begin{enumerate}
 \item[(i)] $(X(r), G_1(r))$, and also  $(Y(r), H_1(r))$($r$ even), have basic normal quotients \\
$(Y(2), H_1(2))$ if $r$ is even and $(X(a), G_1(a))$ for odd primes $a\mid r$;
 \item[(ii)] $(X(r), G_2(r))$, and also  $(Y(r), H_2(r))$, have as basic normal quotients\\ $(Y(a), H_2(a))$ for primes $a\mid 2r$;
 \item[(iii)] $(X(r), G_3(r))$, and also  $(Y(r), H_3(r))$, have as basic normal quotients\\ $(Y(2), H_3(2))$ if $r$ is even and $(Z(a), G_{3Z}(a))$ for odd primes $a\mid r$;
 \item[(iv)] $(Z(s), G_{3Z}(s))$ ($s$ odd) has basic normal quotients $(Z(a), G_{3Z}(a))$ for odd primes $a\mid s$.
\end{enumerate}
}
\end{remark}


\section{Preliminaries: normal quotients of pairs in $\OG(4)$}\label{sec:prelim}

For fundamental graph theoretic concepts please refer to the book \cite{GR}, and for
fundamental notions about group actions please refer to the book \cite{Cam}. 
For permutations $g$ of a set $X$, we denote the image of $x\in X$ under $g$ by $x^g$.

A permutation group
on a set $X$ is \emph{semiregular} if only the identity element fixes a point of $X$; the group is 
\emph{regular} if it is transitive and semiregular.
If a permutation group $G$ on $X$ has a normal subgroup $K\trianglelefteq G$ such that $K$ is regular,
then (see \cite[Section 1.7]{Cam}) $G$ is a semidirect product $G=K.G_x$, where $G_x$ is the stabiliser of a point $x\in X$. 
Moreover, we may identify $X$ with $K$ in such a way that $x=1_K$, $K$ acts by right multiplication and $G_x$ acts by conjugation.
Many of the groups we study in this paper have this form.

As mentioned in Section~\ref{sec:intro}, normal quotients of graph-group pairs $(\Ga,G)\in\OG(4)$  
are usually of the form $(\Ga_N, G/N)$, they lie in $\OG(4)$, and $\Ga$ 
is a cover of $\Ga_N$, where $N \trianglelefteq G$. 
The only exceptions are the degenerate cases when $\Ga_N$ consists of a single vertex  
(if $N$ is transitive), or a single edge (if the $N$-orbits form the bipartition 
of a bipartite graph $\Ga$), or when $\Ga_N$ is a cycle possibly, but not necessarily, 
inheriting a $G$-orientation of its edges, \cite[Theorem 1.1]{janc1}.  
Thus $(\Ga,G)\in\OG(4)$ is \emph{basic} if all of its proper normal quotients
(that is, the ones with $N\ne1$) are degenerate. 

For a graph-group pair $(\Ga,G)\in\OG(4)$ and vertex $x$, the \emph{out-neighbours} of $x$ are the two 
vertices $y$ such that $x\rightarrow y$ is a $G$-oriented edge of $\Ga$.

An \emph{isomorphism} between graph-group pairs $(\Ga,G), (\Ga',G')$ is a pair $(f,\varphi)$ such that 
$f: \Ga\rightarrow \Ga'$ is a graph isomorphism, $\varphi: G\rightarrow G'$ is a group isomorphism,
and $(x^g)f=(xf)^{(g)\varphi}$ for all vertices $x$ of $\Ga$ and all $g\in G$.

\begin{lemma}\label{lem:quotiso} 
Suppose that $(\Ga,G)\in\OG(4)$, $N\trianglelefteq G$, and $(f,\varphi)$ is an isomorphism from 
$(\Ga_N,G/N)$ to $(\Ga',G')$. Let $\mathcal{M}$ be the set of  all normal subgroups $M$ of $G$ such that $N\leq M$ and $M$ is the kernel of
the $G$-action on the $M$-vertex-orbits in $\Ga$. Then, for each $M\in\mathcal{M}$, $(f,\varphi)$ induces an isomorphism
from $(\Ga_M, G/M)$ to $(\Ga'_{\ov M},G'/\ov M)$, where $\ov M = (M/N)\varphi$, and each normal quotient of $(\Ga',G')$ corresponds to exactly one 
such normal quotient $(\Ga_M, G/M)$.
\end{lemma}

\begin{proof}
The isomorphism $\varphi: G/N\rightarrow G'$ determines a one-to-one correspondence $M\mapsto (M/N)\varphi$ between 
the set of  all normal subgroups of $G$ which contain $N$, and the set of all normal subgroups of 
$G'$, and moreover, setting  $\ov M =(M/N)\varphi$, 
$\varphi$ induces an isomorphism $\varphi_M:G/M\rightarrow G'/\ov M$ given by 
$Mg\rightarrow \ov M(Ng)\varphi$. For normal quotient graphs $\Ga_M$, the induced group action is $G/\hat M$,
where $\hat M$ is the kernel of the $G$-action on the $M$-orbits. Thus $\Ga_M=\Ga_{\hat M}$, and the normal quotients of 
$(\Ga,G)$ relative to normal subgroups $M$ containing $N$ are precisely the normal quotients $(\Ga_M,G/M)$, for $M\in\mathcal{M}$.
 
For each $M\in\mathcal{M}$, the graph isomorphism $f:\Ga_N\rightarrow \Ga'$ induces a graph isomorphism $f_M:\Ga_M\rightarrow \Ga'_{\ov M}$,
where $f_M$ maps the $M$-vertex-orbit $x^M$ in $\Ga$ to the $\ov M$-vertex-orbit in $\Ga'$ containing $(x^N)f$.  
By the definition of $(f, \varphi)$ we have 
$((x^N)^{Ng})f=((x^N)f)^{(Ng)\varphi}$ for each vertex $x$ of $\Ga$ and each $g\in G$. It follows that
$((x^M)^{Mg})f_M=((x^M)f_M)^{(Mg)\varphi_M}$ for each vertex $x$ of $\Ga$ and each $g\in G$, and hence
$(f_M,\varphi_M)$ is an isomorphism from $(\Ga_M,G/M)$ to $(\Ga'_{\ov M},G'/\ov M)$.

This property of $(f,\varphi)$ implies that normal subgroups containing $N$ with the same vertex-orbits 
in $\Ga$ correspond to normal subgroups of $G'$ with the same vertex-orbits in $\Ga'$, and vice versa. 
Thus, on the one hand, distinct subgroups in $\mathcal M$ correspond to distinct normal quotients of $(\Ga',G')$. 
Also, on the other hand, each $K\trianglelefteq G'$ such that $K$ is the kernel of the $G'$-action 
on the $K$-vertex-orbits in $\Ga'$ is of the form $K=(M/N)\varphi$, where $N\leq M$ and $M$ is the kernel of
the $G$-action on the $M$-vertex orbits in $\Ga$.   
\end{proof}

\section{Each graph-group pair $(\Ga,G)$ in Theorem~\ref{thm:mainB} lies in $\OG(4)$}\label{sec:xir}  

First we consider $X(r)$ and $G_k(r)$ for $k\leq 3$ and $r\geq 3$. Recall the definitions of the edge orientations and the generators 
$\mu_1,\mu_2,\sigma_1,\sigma_2,\tau$, from Subsection~\ref{sub:results}~(a).
We make a few comments about the various edge orientations.

\begin{remark}\label{rem:xir} {\rm
We view the vertex set of $X(r)$ 
as a $2r\times 2r$ grid with rows and columns labeled $0, 1,\dots,2r-1$ (in this order, increasing to the right and increasing from top to bottom), and with the vertex $(i,j)$ in the \emph{row $i$, column $j$} position. The edges of $X(r)$ then are either \emph{horizontal} if they join vertices with equal first entries, or \emph{vertical} if they join vertices with equal second entries. 
It may be helpful to use this point of view when reading the proofs below. In particular it aids a description of the various edge orientations.

\begin{enumerate}
\item[(1)] In the first edge-orientation, for $i$ even, the horizontal edges in row $i$ are oriented from left to right 
(except of course that the edge from $(i,2r-1)$ is joined to $(i,0)$), and  for $i$ odd, the horizontal edges in row $i$ 
are oriented from right to left. Similarly,  for $j$ even, the vertical edges in column $j$ are oriented downwards and 
those in column $j$, for odd $j$, are oriented upwards.

\item[(2)] In the second edge-orientation, both the rows and the columns are alternating cycles, arranged in such a way that, for each $i,j$, the sequence 
$$(i,j), (i,j+1), (i+1,j+1), (i+1,j)$$ 
forms a directed $4$-cycle, oriented from left to right if $i+j$ is even, and from right to left if $i+j$ is odd. 

\item[(3)] Finally, in the third edge-orientation, all the horizontal edges are oriented from left to right, and all the vertical edges are oriented downwards.

\end{enumerate}
}
\end{remark}

\begin{lemma}\label{lem:x1} 
Let $k\leq 3$ and $r\geq 2$. Then the group $G_k(r)$ preserves the $k^{th}$ edge-orientation of $X(r)$. 
\end{lemma}

\begin{proof}
We consider the cases $k=1, 2, 3$ separately. In each case, for each generator of $G_k(r)$, we consider its 
actions  on a horizontal edge $E$  in row $i$ joining $(i,j)$ to $(i,j+1)$, and on a vertical 
edge $D$ in column $j$ joining $(i,j)$ to $(i+1,j)$. 

(1)  Firstly $\mu_1\sigma_2$ maps $E$ to the horizontal edge $E'$ joining $(i+1,-j)$ to $(i+1,-j-1)$, since $\mu_1\sigma_2$ moves row $i$ to row $i+1$ (by $\mu_1$) and then reflects it across a vertical axis through the $0^{th}$-column (by $\sigma_2$). Thus if $E$ is oriented from left to right, then $E'$ is oriented from right to left (and vice versa), and since horizontal edges in row $i$ and row $i+1$ have opposite orientations (see Remark~\ref{rem:xir}(1)), it follows that $\mu_1\sigma_2$ preserves the orientation of horizontal edges. Also $\mu_1\sigma_2$ maps $D$ to the vertical edge  joining $(i+1,-j)$ to $(i+2,-j)$, and since $j, -j$ have the same parity, edges in columns $j$ and $-j$ have the same orientation (see Remark~\ref{rem:xir}(1)). Thus $\mu_1\sigma_2$ preserves the first edge-orientation on all edges. 
An exactly similar argument (interchanging the roles of rows and columns) shows that $\mu_2\sigma_1$ also preserves the first edge-orientation. 

Finally $\tau$ swaps the horizontal  edge $E$ in row $i$ with the vertical edge $E'$ in column $i$ joining $(j,i)$ to $(j+1,i)$. If $i$ is even then $E$ is oriented from left to right and $E'$ is oriented downwards, so the orientation is preserved. Also, if $i$ is odd then $E$ is oriented from right to left and $E'$ is oriented upwards, and again the orientation is preserved. (For example, the oriented edge $(1,2)\leftarrow (1,3)$ is swapped with the oriented edge $(2,1) \leftarrow (3,1)$.) Similarly the action of $\tau$ preserves the orientation of $D$. Thus $\tau$ preserves the first edge-orientation. 

(2) Firstly $\mu_1^2$ maps $E$ to the horizontal edge $E'$ joining $(i+2,j)$ to $(i+2,j+1)$ and 
$E,E'$ have the same orientation; and $\mu_1^2$ maps $D$ to the vertical edge $D'$ joining $(i+2,j)$ 
to $(i+3,j)$ and $D,D'$ have the same orientation. Thus $\mu_1^2$ preserves the second edge-orientation. 
Similar arguments show that $\mu_2^2$ and $\mu_1\mu_2$ also preserve the second edge-orientation.
Next, $\sigma_1$ maps $E$ to the horizontal edge $E'$ in row $-i$ joining $(-i,j)$ to $(-i,j+1)$ and since $i+j$ and $-i+j$ have the same parity, the edges $E,E'$ have the same orientation; similarly $\sigma_1$ maps $D$ to the vertical edge $D'$, still in column $j$, joining $(-i,j)$ to $(-i-1,j)$ and we see again that $D,D'$ have the same orientation. Thus $\sigma_1$ preserves the second edge-orientation. An exactly similar argument shows that also $\sigma_2$ preserves the second edge-orientation. Finally $\tau\mu_1$ maps $E$ to the vertical edge $E'$ joining $(j+1,i)$ to $(j+2,i)$ and since $i+j$ and $j+1+i$ have opposite parities, the orientation of $E$ is preserved under the action of $\tau\mu_1$; and $\tau\mu_1$ maps $D$ to the horizontal edge $D'$ joining $(j+1,i)$ to $(j+1,i+1)$ and for the same reason  the orientation of $D$ is preserved under the action of $\tau\mu_1$. Thus $\tau\mu_1$ preserves the second edge-orientation.

(3) Firstly, since $\mu_1$ and $\mu_2$ map $E$ to a horizontal edge and $D$ to a vertical edge, 
and since all horizontal edges have the same orientation, and all vertical edges have the same 
orientation (see Remark~\ref{rem:xir}(3)), it follows that $\mu_1$ and $\mu_2$ preserve the 
third edge-orientation. Finally $\tau$ swaps horizontal and vertical edges, and since all 
horizontal edges are oriented from left to right, and all vertical edges are oriented downwards, 
it follows that $\tau$ also preserves the third edge-orientation.
\end{proof}

Next we prove membership of $\OG(4)$. Recall that, for a subgroup $L\leq G(r)$ we write $\ov L = LM(r)/M(r)$ for the corresponding subgroup of $H(r)=G(r)/M(r)$.

\begin{lemma}\label{lem:x2} 
Let $k\leq 3$, $r\geq 2$, and let $s>1$ be odd. Let $(\Ga,G)$ be one of $(X(r),G_k(r))$, $(Y(r),H_k(r))$ (with $r$ even if $k=1$), or $(Z(s),G_{3Z}(s))$.
Then $(\Ga,G)\in\OG(4)$.
\end{lemma}
 

\begin{proof}
It is easy to check that the graphs $X(r)$ and $Z(s)$ are connected. 
Thus,  once we have proved that $G_k(r)$ acts half-arc-transitively on $X(r)$, and $G_{3Z}(s)$ 
 acts half-arc-transitively on $Z(s)$, it follows from Lemma~\ref{lem:x1} 
that  $(X(r),G_k(r))\in\OG(4)$ and   $(Z(s),G_{3Z}(s))\in\OG(4)$. Further, 
as $(Y(r),H_k(r))$ is a normal quotient of $(X(r),G_k(r))$ relative to $M(r)\cong Z_2$, 
its  membership in $\OG(4)$ follows from \cite[Theorem 1.1]{janc1}. 
It therefore remains to prove that the actions on $X(r)$ and $Z(s)$
are half-arc-transitive. First we consider $X(r)$. 

The subgroup $L:=N(2) = \la{\mu_1}^2\ra\times \la{\mu_2}^2\ra\leq G(r)$ is contained in $G_k(r)$, for each $k$, and the $L$-orbits on vertices are the following four subsets.
\begin{align}\label{ee} 
 {\Delta}_{ee} &= \{ (i,j)\ |\  i, j \hspace{.2cm} \mbox{are both even} \},\quad &{\Delta}_{eo} &= \{ (i,j)\ |\  i \hspace{.2cm}\mbox{is even}, \hspace{.2cm} j \hspace{.2cm} \mbox{is odd}\},\\
{\Delta}_{oe} &= \{ (i,j)\ |\ i \hspace{.2cm}\mbox{is odd}, \hspace{.2cm} j \hspace{.2cm} \mbox{is even}\},\quad &{\Delta}_{oo} &= \{ (i,j)\ |\ i, j \hspace{.2cm} \mbox{are both odd} \}. \nonumber
\end{align}
Writing $\bar g = Lg$ for each $g\in G(r)$, we note that  $G(r)/L = \la \bar\mu_1,\bar\mu_2,\bar\sigma_1,\bar\sigma_2,\bar\tau\ra \cong (Z_2\times Z_2)\wr Z_2$ of order $2^5$. 
Let $x:=(0,0)\in \Delta_{ee}$.
For each $k$ we find a normal subgroup $K$ of $G_k(r)$ such that $K$ contains $L$ and $K$ is regular on vertices. Hence $G_k(r)$ is the semidirect product $G_k(r)=K\cdot(G_k(r))_x$.
Then to prove  half-arc-transitivity, it is sufficient to prove that $(G_k(r))_x$ fixes setwise and interchanges the two out-neighbours of $x$. 
(This is sufficient since, for two arbitrary oriented edges, say $y\rightarrow z$ and $y'\rightarrow z'$, there are elements $g, g'\in K$ such that $y^g=x$, $(y')^{g'}=x$, so that $z^g, (z')^{g'}$ 
are possibly equal out-neighbours of $x$.)

\smallskip
(1) For the first edge-orientation we see that $K:=\la \mu_1\sigma_2,\mu_2\sigma_1\ra$ is normalised by $\tau$ and hence 
$K$ is normal in $G_1(r)$ with index 2, and $G_1(r)=K\cdot\la\tau\ra$. Since $|G_1(r)|=8r^2$ 
(see Subsection~\ref{sub:results}), we have $|K|=4r^2$. Moreover, since  
$x^{\mu_1\sigma_2} = (1,0)\in\Delta_{oe},$ $x^{\mu _2\sigma _1} = (0,1)\in\Delta_{eo}$, and $(0,1)^{\mu_1\sigma_2} = (1,-1)\in\Delta_{oo}$, 
it follows that $K$ permutes the $L$-orbits transitively, and hence $K$ is transitive on vertices. Now $|K|=4r^2$, and hence 
$K$ is regular on vertices. Thus the stabiliser $(G_1(r))_x$ has order 2 and so is equal to $\la\tau\ra\cong Z_2$. 
Finally $\tau$ interchanges the two out-neighbours $(0,1)$ and $(1,0)$ of $x$.

\smallskip
(2) For the second edge-orientation, the subgroup $K:=\la \mu_1\mu_2, \mu_1^2, \tau\mu_1\ra$ is normalised by $\sigma_1$ 
and $\sigma_2$ and hence is normal in $G_2(r)$. Also the subgroup $H:=\la\sigma_1,\sigma_2\ra\cong Z_2^2$ fixes the vertex 
$x$ and $H\cap K=1$. Since $|G_2(r)|=16r^2=|K|.|H|$, this  implies that $G_2(r)$ is the semidirect product $K\cdot H$,
and $|K|= 16r^2/|H|=4r^2$.  
Since $x^{\mu_1\mu_2} = (1,1)\in\Delta_{oo},$ $x^{\tau\mu_1} = (1,0)\in\Delta_{oe}$, and $(1,0)^{\mu_1\mu_2} = (2,1)\in\Delta_{eo}$, 
it follows that $K$ permutes the $L$-orbits transitively, and hence $K$ is transitive on vertices. 
Then since $|K|=4r^2$, it follows that $K$ is regular on vertices. Thus the stabiliser $(G_2(r))_x$ has order 
$|G_2(r):K|=4$ and so is equal to $H$. Finally $\sigma_2\in H$ interchanges the two out-neighbours $(0,1)$ and $(0,-1)$ of $x$.
 
\smallskip
(3) For the third edge-orientation, the subgroup $K:=\la \mu_1, \mu_2\ra$ is normalised by $\tau$ and hence is normal 
in $G_3(r)$ of index $2$, and $G_3(r)=K\cdot\la\tau\ra$ so $|K|=4r^2$. Moreover, since $x^{\mu_1} = (1,0)\in\Delta_{oe},$ 
$x^{\mu_2} = (0,1)\in\Delta_{eo}$, and $(0,1)^{\mu_1} = (1,1)\in\Delta_{oo}$, it follows that $K$ permutes the $L$-orbits 
transitively, and hence $K$ is transitive on vertices. 
Now $|K|=4r^2$, and hence $K$ is regular on vertices. Thus the stabiliser $(G_3(r))_x$ has order $2$ and so is equal to 
$\la\tau\ra\cong Z_2$. Finally $\tau$ interchanges the two out-neighbours $(0,1)$ and $(1,0)$ of $x$.
\smallskip

(4) Now we consider $Z(s)$ and  $G=G_{3Z}(s)$ with $s$ odd. It is straighforward to show that the subgroup $K=\la\mu_1,\mu_2\ra$ 
is regular on vertices, and $G=K.\la\tau\ra$ with $G_x=\la\tau\ra$. Also  $\tau$ interchanges the two out-neighbours $(0,1)$ and $(1,0)$ of $x$.
\end{proof}

\section{Classifying the edge-orientations}\label{sub:allorients}

In this section we prove Theorem~\ref{thm:allorients}. Suppose that $\Ga$ is one of the graphs 
$X(r), Y(r), Z(s)$ defined in Subsection~\ref{sub:results}, and that $H\leq\Aut\,\Ga$ acts half-arc-transitively. 
Then as discussed in Subsection~\ref{sec:alt}, $H$ preserves an edge-orientation of $\Ga$, and this edge-orientation is 
determined by $H$ up to reversing the orientation on each edge. Moreover, this edge-orientation and its `reverse' 
will be the same as those preserved by a subgroup of $\Aut\,\Ga$ which is maximal subject to containing $H$ and acting 
half-arc-transitively on $\Ga$. Thus proving  Theorem~\ref{thm:allorients} is equivalent to classifying all of the 
subgroups $H$ of $\Aut\,\Ga$ which are maximal subject to acting half-arc-transitively on $\Ga$. 
First we show that a proof of Theorem~\ref{thm:allorients} in
the case of $X(r)$ implies the result for the graph $Y(r)$. 


\begin{lemma}\label{lem:Yaut}
 If  $r\not\in\{1,2,4\}$, then $\Aut\,Y(r)$ is the group $G(r)/M(r)$ discussed
in Subsection~$\ref{sub:results} (b)$. Moreover, if the assertions of Theorem~$\ref{thm:allorients}$ 
hold for $X(r)$, then they hold also for $Y(r)$.
\end{lemma}

\begin{proof}
Write $\overline{G(r)}=G(r)/M(r)$, so  $\overline{G(r)}$ is an arc-transitive subgroup of 
$A:=\Aut\,Y(r)$, where $M(r)=\la (\mu_1\mu_2)^r\ra$ as in Subsection~\ref{sub:results}. 
We consider the vertex $\overline{x}=\{(0,0),(r,r)\}$ of $Y(r)$, and its four neighbours 
$\overline{y}=\{(1,0),(r+1,r)\}$, $\overline{y'}=\{(-1,0), (r-1,r)\}$, $\overline{z}=\{(0,1), (0,r+1)\}$ and 
$\overline{z'} = \{(0,-1),(r,r-1)\}$.  The stabiliser in $\overline{G(r)}$ of the arc $(\overline{x}, \overline{y})$ 
is the subgroup $\overline{\la\sigma_2\ra}$. Since  $r\not\in\{1,2,4\}$, there are two paths of length $2$
joining $\overline{y}$ to the vertex $v$, for $v=\overline{z}$ and $v=\overline{z'}$ but only 
one such path for $v=\overline{y'}$. It follows that $A_{\overline{x}, \overline{y}}$ must also fix 
$\overline{y'}$. 

Thus $A_{\overline{x},\overline{y},\overline{z}}$ has index $2$ in $A_{\overline{x},\overline{y}}$,
and fixes each of the four neighbours of $\overline{x}$.
Moreover, $A_{\overline{x},\overline{y},\overline{z}}$ fixes the second common neighbour 
$\{(1,1),(r+1,r+1)\}$ of $\overline{y}$ and $\overline{z}$, and also the second common neighbour 
$\{(1,-1),(r+1,r-1)\}$ of $\overline{y}$ and $\overline{z'}$. It follows that 
$A_{\overline{x},\overline{y},\overline{z}}$ fixes pointwise each of the neighbours of
$\overline{y}$. Repeating this argument we conclude that $A_{\overline{x},\overline{y},\overline{z}}=1$.
Thus $|A|$ is equal to twice the number of arcs, and hence $A=\overline{G(r)}$. 

To prove the last assertion, suppose that $\overline{H}\leq A$ and $\overline{H}$ is maximal subject to being
half-arc-transitive on $Y(r)$, and thus preserving a certain edge-orientation of $Y(r)$ (and its reverse edge-orientation).
Define an edge-orientation of $X(r)$ by orienting each edge of $X(r)$ according to 
the orientation of the corresponding edge of $Y(r)$. Since $X(r)$ is a double 
cover of $Y(r)$, this means that both $X(r)$-edges corresponding to a $Y(r)$-edge will have the same orientation.   
Now $\overline{H} = H/M(r)$ where $M(r)\leq H\leq G(r)$,
$H$ is half-arc-transitive on $X(r)$, and $H$ preserves this edge-orientation of $X(r)$. 
Thus, by assumption, replacing $H$ by a conjugate in $\Aut\,X(r)$ if necessary, 
this edge-orientation of $X(r)$ is one of the three defined in Subsection~\ref{sub:results}~(a)
or the reverse of one of these, and so $H\leq G_k(r)$ for some $k\in\{1,2,3\}$. 
Moreover if $k=1$, then $r$ is even, since in this case, if $r$ were odd then the two $X(r)$-edges 
corresponding to a $Y(r)$-edge would have opposite orientations (see Remark~\ref{rem:xy}~(a)). 
\end{proof}

By Lemma~\ref{lem:Yaut}, we may assume that $\Ga=X(r)$ or $Z(s)$, or in other words, that $\Ga=
C_t\,\square\, C_t$, where either $t=2r$, or $t=s$ (with $s$ odd).
Then $A:=\Aut\,\Ga=\la \mu_1,\mu_2,\sigma_1,\sigma_2,\tau\ra$, where the generators are defined as in
Subsection~\ref{sub:results}~(a). We use this notation for the rest of this subsection 
and we assume that  $H$ is a half-arc-transitive subgroup of $\Aut\,\Ga$. 
We let $x=(0,0)$, and 
denote its neighbours in $\Ga$ by $y=(1,0)$, $y'=(-1,0)$, $z=(0,1)$ and $z'=(0,-1)$.
First we obtain some restrictions on $H$.

\begin{lemma}\label{lem:h}
 \begin{enumerate}
  \item[(a)]  The group $H$ contains none of the elements $
\sigma_1\mu_1,\ \sigma_2\mu_2,$ $\sigma_1\sigma_2\mu_1,$ or $\sigma_1\sigma_2\mu_2.$

\item[(b)] Up to conjugation in $A$, the stabiliser $H_x$ is one of 
$\la\sigma_1\sigma_2\ra$, $\la\sigma_1,\sigma_2\ra$ or $\la\tau\ra$. 
 \end{enumerate}
\end{lemma}

\begin{proof}
 As noted in Subsection~\ref{sub:results}, $A_x = \la\sigma_1,\sigma_2,\tau\ra\cong D_8$,
and $A_{x,y}=\la\sigma_2\ra\cong Z_2$. There are thus exactly two elements of $A$ which reverse 
the arc $(x,y)$, and an easy computation shows these are the elements $\sigma_1\mu_1$
and $\sigma_1\sigma_2\mu_1$. Since $H$ is vertex-transitive, and edge-transitive, but 
not arc-transitive, these two elements do not lie in $H$. The same argument for the arc 
$(x,z)$ shows that $H$ does not contain  $\sigma_2\mu_2$
or $\sigma_1\sigma_2\mu_2$. This proves part (a).

Since $H$ is not arc-transitive, it is a proper subgroup of $A$, so $H_x$ 
is a proper subgroup of $A_x$, and $H_x$ has two orbits on $\{y,y'z,z'\}$,
the set of neighbours of $x$. If $|H_x|=4$ these properties imply that 
$H_x=\la\sigma_1,\sigma_2\ra$. The only other possibility is that $|H_x|=2$, 
so suppose this is the case. Of the five involutions in $A_x$ the only ones
which act on $\{y,y'z,z'\}$ with two cycles of length $2$ are $\sigma_1\sigma_2, 
\tau$ and $\sigma_1\sigma_2\tau$. Since $\tau^{\sigma_1} = \sigma_1\sigma_2\tau$, 
part (b) follows.
\end{proof}

This result allows us, in particular, to deal with the case where $H$ contains 
the subgroup $M:=\la\mu_1,\mu_2\ra$.

\begin{lemma}\label{lem:z}
If $H$ contains $M:=\la\mu)1,\mu_2\ra$, then up to conjugation in $A$,
\[
 H = \begin{cases}
      G_3(r) & \mbox{if $t=2r$}\\
      G_{3Z}(s) & \mbox{if $t=s$ is odd.}
     \end{cases}
\]
  In particular the edge-orientation is the one defined in Subsection~$\ref{sub:results} (a.3)$ or $(a.4)$,
or its reverse.  Moreover, all assertions in Theorem~\ref{thm:allorients} hold if $\Ga=Z(s)$.
\end{lemma}

\begin{proof}
Suppose that $M\leq H$. Then $H=M\rtimes H_x$, as $M$ is regular on vertices, and it follows from 
Lemma~\ref{lem:h} that $H_x=\la\tau\ra$ (replacing $H$ by a conjugate if necessary). Thus $H$ is as claimed
and so is the edge-orientation. If $\Ga=Z(s)$ then $|A|=8s^2$ and the subgroup $M$ is the unique Hall 
$2'$-subgroup of $A$. Since $H$ is vertex transitive $H$ must contain $M$.
\end{proof}

We may therefore assume that $\Ga=X(r)$, and that $M\not\subseteq H$. Next we consider the case $\tau\in H$.

\begin{lemma}\label{lem:tau}
If $\Ga=X(r)$,  $M\not\subseteq H$, and $\tau\in H$, then up to conjugation in $A$,
$H=G_1(r)$ and the edge-orientation is as defined in Subsection~$\ref{sub:results} (a.1)$ or its reverse. 
\end{lemma}

\begin{proof}
Since $\tau\in H$, we have $H_x=\la\tau\ra$ by Lemma~\ref{lem:h}, and so $|A:H|=4$.
Also $|M:M\cap H|=|MH:H|$ divides $|A:H|=4$.

We claim that $|M:M\cap H|=4$. Suppose not. Then since  $M\not\subseteq H$, the subgroup $M\cap H$ 
has index $2$ in $M$, and as it is $\tau$-invariant, it follows that $M\cap H=\la\mu_1^2, \mu_2^2, \mu_1\mu_2\ra$. 
Now $MH/M \cong H/(M\cap H)$ has order $|H|/|M\cap H|=4$, and as $A_x$ is a transversal for $M$ in $A$, 
the $M$-cosets in $MH/M$ have representatives from a subgroup of $A_x$ of order $4$ containing $\tau$. 
The unique such subgroup is $\la\tau,\sigma_1\sigma_2\ra$. Since each of these $M$-cosets also has 
a representative from $H$, it follows that $H$ contains an element $\sigma_1\sigma_2\mu_1^i\mu_2^j$,
for some $i, j$. Since $H$ contains $\mu_1^2,\mu_2^2$ and $\mu_1\mu_2$, we may assume that $j=0$ and
$i\in\{0,1\}$. Part (a) of Lemma~\ref{lem:h} implies that $i\ne1$, while the fact that $H_x=\la\tau\ra$ 
implies that $i\ne 0$.  This is a contradiction, and hence  $|M:M\cap H|=4$.
Thus $|MH/M|=|H|/|M\cap H|=8$, so $MH=A$.

We next claim that $M\cap H=\la\mu_1^2,\mu_2^2\ra$. Since $M\cap H$ is $\tau$-invariant, this claim 
would follow if the projection of $M\cap H$ into $\la\mu_1\ra$ were contained in $\la\mu_1^2\ra$
(since this would imply that $M\cap H\subseteq \la\mu_1^2,\mu_2^2\ra$).  Suppose that this is not the case.
Then $M\cap H$ is a $\tau$-invariant, subdirect subgroup of $M=\la\mu_1\ra\times\la\mu_2\ra$ 
of index $4$, and it follows that $r$ is even and $M\cap H=\la \mu_1^4, \mu_2^4, \mu_1\mu_2^a \ra$, where 
$a\in\{1, -1\}$. Since $MH=A$, $H$ contains an element of the form $h=\sigma_2\mu_1^i\mu_2^j$ for some $i, j$,
and hence $H$ contains $\mu_1\mu_2^a(\mu_1\mu_2^a)^h = \mu_1\mu_2^a\mu_1\mu_2^{-a} = \mu_1^2$.
This is a contradiction, proving our second claim. 
 
Thus $M\cap H=\la\mu_1^2,\mu_2^2\ra$. Since $MH=A$, the subgroup $H$ contains an element $h'=\sigma_1\mu_1^i\mu_2^j$ 
for some $i, j$, and since $M\cap H$ contains $\mu_1^2, \mu_2^2$, we may assume that $i, j\in\{0,1\}$.
Since $\sigma_1\not\in H$, and also, by Lemma~\ref{lem:h}, $\sigma_1\mu_1\not\in H$, it follows that $j=1$.
If $h'=\sigma_1\mu_1\mu_2$, then $H$ also contains $(h')^\tau = \sigma_2\mu_2\mu_1$, and this implies that
$H$ contains $(h')^\tau(h')^{-1}=\sigma_2\sigma_1=\sigma_1\sigma_2$, which is a contradiction.
Thus $h'=\sigma_1\mu_2$. Then $H$ also contains $(h')^\tau = \sigma_2\mu_1$, and so the group $G_1(r)$ 
is contained in $H$. These groups have the same order, so $H=G_1(r)$ and the lemma follows.
\end{proof}

Thus from now on we may assume that $\tau\not\in H$ and hence, by Lemma~\ref{lem:h}, that
$H_x=\la\sigma_1\sigma_2\ra$ or $\la\sigma_1, \sigma_2\ra$.

\begin{lemma}\label{lem:sigsig}
If $\Ga=X(r)$,  $M\not\subseteq H$, and $H_x=\la\sigma_1, \sigma_2\ra$ or $\la\sigma_1\sigma_2\ra$, 
then up to conjugation in $A$,
$H\leq G_2(r)$ and the edge-orientation is as defined in Subsection~$\ref{sub:results} (a.2)$ or its reverse. 
\end{lemma}

\begin{proof}
 Here $H_x=\la\sigma_1\sigma_2\ra$ or $\la\sigma_1, \sigma_2\ra$ and $|H|=8r^2$ or $16r^2$ respectively.
Let $\delta:=|M:M\cap H|$, so $\delta>1$ since $M\not\subseteq H$. 
Then $|H| = |H:M\cap H|\cdot |M\cap H|= |MH:M|\cdot (4r^2/\delta)$ divides $8\cdot (4r^2/\delta)$,
so $\delta \mid 4$ or $\delta=2$ according as  $|H|=8r^2$ or $16r^2$. Note that it is sufficient 
to prove that $H\leq G_2(r)$ (up to conjugation in $A$), since the edge-orientations preserved by 
$H$ and $G_2(r)$ are then the same as both act half-arc-transitively.
Let $\pi_i:M\rightarrow \la\mu_i\ra$ denote the natural projection map, for $i=1, 2$. 

\smallskip\noindent
\emph{We claim that 
$\pi_i(M\cap H) = \la\mu_i\ra$ for at least one $i$.} Suppose that this does not hold. Then, 
since $\delta\leq 4$, it follows that $M\cap H = \la \mu_1^2,\mu_2^2\ra$, $\delta=4$, so 
$H_x=\la\sigma_1\sigma_2\ra$ and $A=MH$. Since $A=MH$ it follows that $H$ contains elements of the 
form $h=\tau\mu_1^i\mu_2^j$ and $h'=\sigma_1\mu_1^{i'}\mu_2^{j'}$, and we may assume that $i,j,i',j'$ 
all lie on $\{0,1\}$ since $\mu_1^2, \mu_2^2$ lie in $H$. Since neither $\tau$ nor $\sigma_1$ lies in $H$, 
$(i,j)\ne (0,0) \ne (i',j')$. By Lemma~\ref{lem:h}~(a), $H$ does not contain either $\sigma_1\mu_1$
or $\sigma_2\mu_2 = (\sigma_1\sigma_2)(\sigma_1\mu_2)$, and hence the only possiblity for $h$ is
$\sigma_1\mu_1\mu_2$. If $h'=\tau\mu_1$ then $H$ contains $(h')^2 = \mu_2\mu_1$, which 
is a contradiction. Similarly $h'\ne \tau\mu_2$. Thus $h'=\tau\mu_1\mu_2$, but then $H$ contains
$h'h^{-1} = \tau\sigma_1$, which is a contradiction, proving the claim.

\smallskip
Replacing $H$ by its conjugate $H^\tau$ if necessary, we may assume that $M\cap H$ 
contains $\mu_1\mu_2^a$ for some $a$. Then $|M\cap H|=2r\,|H\cap\la\mu_2\ra|$,
and hence $H\cap \la\mu_2\ra=\la\mu_2^\delta\ra$ and
$M\cap H = \la\mu_1\mu_2^a, \mu_2^\delta\ra$, where $r$ is even if $\delta=4$. 
We may also assume that $0\leq a<\delta$. Moreover $a\ne0$, since otherwise 
$H$ contains $(\sigma_1\sigma_2)\mu_1$,
contradicting Lemma~\ref{lem:h}. 

\smallskip\noindent\emph{We claim that $\delta=2$.}
Suppose to the contrary that $\delta=4$. Then  $r$ is even, $H_x=\la\sigma_1\sigma_2\ra$,  
and $M\cap H = \la\mu_1\mu_2^a, \mu_2^4\ra$, with $a\in\{ 1, 2, 3\}$. 
The equations in the first paragraph imply that $|MH:M|=8$ so $A=MH$, and hence
$H$ contains elements of the form $h=\sigma_1\mu_2^b$ and $h'=\tau\mu_2^c$ for some $b, c\in\{1,2,3\}$
(adjusting by elements of $M\cap H$ and noting that $\tau, \sigma_1\not\in H$). Then
$M\cap H$ contains $\mu_1\mu_2^a (\mu_1\mu_2^a)^h = \mu_2^{2a}$, which implies that $a=2$,
and hence that $M\cap H$ contains $\mu_1^2$.  
Therefore  $M\cap H$ contains $(\mu_1\mu_2^2)^{h'}\mu_1^{-2}=\mu_2$, which is a contradiction.

\smallskip
Thus $\delta=2$, so $M\cap H = \la\mu_1\mu_2, \mu_2^2\ra$, 
and the equations in the first paragraph imply that $|MH:M|=4$ or $8$, when $|H|=8r^2$ or $16r^2$ respectively. 
Then $MH/M$ has order at least $4$. Suppose that $H$ contains an element of the form 
$h=\tau\mu_1^i\mu_2^j$; adjusting by an element of 
$M\cap H$ we may assume that $h=\tau \mu_1$ (since $H$ does not contain $\tau$). Then 
$H=\la M\cap H, H_x,\tau\mu_1\ra \leq \la \mu_1\mu_2, \mu_2^2, \sigma_1, \sigma_2, \tau\mu_1\ra
= G_2(r)$, and the lemma is proved in this case. If $H$ contains no such element then $|MH/M|=4$.
Next, if $H$ contains an element of the form $h=\tau\sigma_\ell\mu_1^i\mu_2^j$ (for $\ell=1$ or $2$), then
adjusting by an element of $M\cap H$ we may assume that $h=\tau\sigma_\ell \mu_1$ 
(since $H$ does not contain $\tau\sigma_\ell$), and again we find that $H=\la M\cap H, H_x,
\tau\sigma_\ell\mu_1\ra \leq G_2(r)$. The lemma follows in this case.  Thus we may assume that 
$MH/M$ projects to the subgroup $\la\sigma_1,\sigma_2\ra$ of $A_x$, and now
we obtain an element of the form $h=\sigma_1\mu_1^i\mu_2^j$ in $H$, and
$H=\la M\cap H, H_x, h\ra\leq G_2(r)$, completing the proof.
 \end{proof}

All the assertions of Theorem~\ref{thm:allorients} now follow from Lemmas~\ref{lem:Yaut}, 
\ref{lem:z}, \ref{lem:tau}, and~\ref{lem:sigsig}, completing its proof.

\section{Identifying normal quotients of $(\Ga,G)$ for Theorem~\ref{thm:main}}\label{sec:normalquots}  

We identify some of the normal quotients of these graph-group pairs. 
Note that whenever a normal subgroup $N$ of $G(r)$ is contained in $G_k(r)$ we can use it to form a normal quotient of $(X(r), G_k(r))$,
and moreover we can use Lemma~\ref{lem:quotiso}  to deduce information about normal quotients of $(Y(r), H_k(r))$ (taking $N=M(r)$) 
and about $(Z(s), G_{3Z}(s))$ (taking $r=s$ odd and $N=N(s)$).    
Recall the definitions of the subgroups in \eqref{na}, \eqref{ld} and \eqref{ld+}. First we deal with normal quotients modulo $N(a), M(a)$, 
and the corresponding subgroups of $H(r)=G(r)/M(r)$. 

\begin{lemma}\label{lem:lines1-4} 
For $\Ga, G, N$ as in one of the lines of Table~$\ref{tbl:main}$, the assertions about the normal quotient $(\Ga_N, G/N)$ are valid. 
\end{lemma}

\begin{proof}
Let $G=G_k(r)$ or $G_{3Z}(s)$, and let $\Ga=X(r)$ or $Z(s)$, so $\Ga$ has vertex set $\mathbb{Z}_t\times\mathbb{Z}_t$, where $t=2r$ or $s$ respectively.
Consider the action on vertices of the normal subgroup $N=N(a)$ of $G$, where $a\mid t$ and $2<a<t$.
The $N$-orbits are the $a^2$ subsets
\[
\Delta_{i,j} =\{(i' j')\mid i'\equiv i\pmod{a},\ j'\equiv j\pmod{a}\} 
\]
for $i,j\in\mathbb{Z}_{a}$. Since $(i',j')$ is adjacent in $\Ga$ to $(i'\pm1, j')$ and 
$(i', j'\pm 1)$, for all $i',j'\in\mathbb{Z}_{t}$, it follows that $\Delta_{i,j}$ is 
adjacent in the quotient graph $\Ga_{N}$ to $\Delta_{i\pm1,j}$ and $\Delta_{i,j\pm1}$, 
for each $i,j\in\mathbb{Z}_{a}$. Thus, since $a>2$, $\Ga_{N}$  has valency $4$, and the mapping 
$f:\Delta_{i,j}\longrightarrow (i,j)$ defines a graph isomorphism from $\Ga_N$ to $X(a/2)$ if $a$ is 
even, or to $Z(a)$ if $a$ is odd. 
By \cite[Theorem 1.1]{janc1}, the group induced by $G$ on the quotient $\Ga_{N}$ is precisely $G/N$, and  
$(\Ga_{N},G/N)\in\OG(4)$. In particular $N$ is the kernel of the $G$-action on the set of $N$-orbits in $\Ga$.
Write $\bar g:= Ng$ for elements of the quotient group $G/N$. 

If $a$ is even, then it follows from the definitions 
of the generators $\mu_1,\mu_2,\sigma_1,\sigma_2,\tau$ of $G(r)$, given in Subsection~\ref{sub:results},  
that the induced maps $\bar\mu_1, \bar\mu_2, \bar\sigma_1, \bar\sigma_2, \bar\tau$ acting on $\Ga_N$ 
correspond to the respective generators for the smaller group $G(a/2)$ acting on $X(a/2)$. This natural 
correspondence defines a group isomorphism $G(r)/N\rightarrow G(a/2)$ which restricts to an 
isomorphism $\varphi: G/N\rightarrow G_k(a/2)$. We conclude that $(f,\varphi)$ defines an isomorphism 
from $(\Ga_{N},G/N)$ to $(X(a/2), G_k(a/2))$. Thus the first line of Table~\ref{tbl:main} is valid.

Similarly if $a$ is odd, then the maps $\bar\mu_1, \bar\mu_2, \bar\tau$ acting on $\Ga_N$
induced from the generators $\mu_1,\mu_2,\tau$ (for $G_3(r)$ or $G_{3Z}(s)$) correspond to the respective generators 
for the smaller group $G_{3Z}(a)$ acting on $Z(a)$, and we obtain  an isomorphism 
from $(\Ga_{N},G/N)$ to $(Z(a), G_{3Z}(a))$. Thus lines 3 and 7 of Table~\ref{tbl:main} are also valid.

Suppose now that $t=2r$ above, so $\Ga= X(r)$. If $a$ and $\frac{2r}{a}$ are both even then by \eqref{nm}, $N(a)$ contains $M(r)$, and it follows 
by Lemma~\ref{lem:quotiso} (taking $N=M(r)$ in that result) that the quotient of $(Y(r),H_k(r))$ modulo $\ov{N(a)}$ is isomorphic to
the quotient of $(X(r),G_k(r))$ modulo $N(a)$, and we have just shown that this latter quotient 
is isomorphic to $(X(a/2), G_k(a/2))$. Thus line 4 of Table~\ref{tbl:main} is valid. Similarly if $a$ is odd then 
$\frac{2r}{a}$ is even, and again by \eqref{nm}, $N(a)$ contains $M(r)$. The same argument now yields that
the quotient of $(Y(r),H_k(r))$ modulo $\ov{N(a)}$ is isomorphic to $(Z(a), G_{3Z}(a))$, proving that
line 6 of Table~\ref{tbl:main} is valid.

It remains to consider lines 2 and 5 of Table~\ref{tbl:main}. We continue to let $(\Ga,G)$ 
be the pair $(X(r),G_k(r))$,
and we note that, if a  normal quotient of $(\Ga,G)$ modulo $M$ is 4-valent, then by 
\cite[Theorem 1.1]{janc1}, $M$ is the kernel of the $G$-action on the set of $M$-vertex-orbits in $\Ga$. 
Consider now $M=M(a)$ where $a\mid r$ (so $\frac{2r}{a}$ is even) and $1<a\leq r$. 
Since $M\subseteq G=G_k(r)$, we must have $a$ even when $k=1$. 
Applying Lemma~\ref{lem:quotiso} with $(\Ga',G')=(X(a),G_k(a))$ and $N=N(2a)$, we find that
the quotient of $(\Ga,G)$ modulo $M$ is isomorphic to the quotient of $(X(a),G_k(a))$ modulo
the image of $M$ under projection from $G$ to $G/N\cong G_k(a)$, namely $\la (\bar\mu_1\bar\mu_2)^a\ra$.
Thus the latter normal quotient is $(Y(a),H_k(a))$, proving that line 2 of Table~\ref{tbl:main} is valid. 

For the final line, line 5, we apply  Lemma~\ref{lem:quotiso} with $(\Ga',G')=(Y(r),H_k(r))$ and $N=M(r)$, where $r$ is even when $k=1$.
Consider   $M=M(a)$,  where $a\mid r$ and $1<a< r$, and note that $N\leq M$. We wish to take the  quotient of $(\Ga',G')$ modulo 
$\ov M = M(a)M(r)/M(r)$, and we note that $1<\ov M\leq H_k(r)$ if and only if $a<r$, and also $a$ is even when $k=1$.  
Suppose this is the case. Then by Lemma~\ref{lem:quotiso}, 
the quotient of $(\Ga,G)$ modulo $M$ is isomorphic to the quotient of $(Y(r),H_k(r))$ modulo
the image $\ov M$ of $M(a)$ under projection from $G$ to $G/N\cong H_k(r)$. We already proved that the 
former normal quotient $(\Ga_M,G/M)$ is isomorphic to $(Y(a), H_k(a))$. This proves 
that line 5 of Table~\ref{tbl:main} is valid. 
\end{proof}

Next we consider normal quotients in Table~\ref{tbl:main2}. Recall the definitions of the subgroups in \eqref{na}, \eqref{ld}  and \eqref{ld+}.

\begin{lemma}\label{lem:tbl2} 
For $\Ga, G$ as in Theorem~$\ref{thm:main}(b)$, and for $N$ as in one of the lines of Table~$\ref{tbl:main2}$, the assertions about the normal quotient $(\Ga_N, G/N)$ are valid. 
\end{lemma}

\begin{proof}
The $N(2)$-vertex-orbits on $X(r)$ are the four subsets given in \eqref{ee}. It is straightforward 
to check that the underlying quotient graph of $X(r)$ modulo $N(2)$ is a cycle $C_4$. With the first and 
third edge-orientations (that is, $k=1$ or $3$), there are oriented edges 
in both directions between the adjacent $N(2)$-orbits, so the cyclic quotient is unoriented, 
$N(2)$ is the kernel of the action on the $N(2)$-vertex-orbits, and the induced group is $D_8$. Thus line 1 of 
Table~\ref{tbl:main2} is valid for $N$ in $G_k(r)$ with $k=1,3$. 
Moreover in these cases, if $r$ is even, then $N(2)$ contains $M(r)$ by \eqref{nm}, and it follows from Lemma~\ref{lem:quotiso}
that the quotient of $(Y(r),H_k(r))$ modulo $\ov{N(2)}$ is isomorphic to the quotient of $(X(r),G_k(r))$ modulo $N(2)$, namely $(C_4,D_8)$,
proving the rest of  line 1 of Table~\ref{tbl:main2}. 

On the other hand, if $k=2$, then the edges of the quotient graph of $X(r)$ modulo $N(2)$ (which is a 4-cycle) are oriented 
$\Delta_{ee}\rightarrow\Delta_{eo}\rightarrow\Delta_{oo}\rightarrow\Delta_{oe}\rightarrow\Delta_{ee}$, the 
kernel of the $G_2(r)$-action on the $N(2)$-vertex-orbits is $N(2,+)$, and $G_2(r)/N(2,+)\cong Z_4$. 
Thus line 2 of Table~\ref{tbl:main2} is valid for $N$ in $G_2(r)$. If  $r$ is even, then 
$M(r)\leq N(2,+)$ and arguing as in the previous paragraph, the normal quotient of 
$(Y(r),H_2(r))$ modulo $\ov{N(2,+)}$ is also $(C_4,Z_4)$, proving the rest of  line 2 of 
Table~\ref{tbl:main2}.  

Next suppose that $r$ is odd and $k=2$, or $3$.  Then the preimage in $G_k(r)$ of  $\ov{N(2,+)}$, or $\ov{N(2)}$, is equal to 
$\la N(2,+),\mu_1\mu_2\ra$, or $M(1)$, respectively. Each of these preimage groups has vertex-orbits $\Delta_{ee}\cup\Delta_{oo}$ and 
$\Delta_{eo}\cup\Delta_{oe}$ in $X(r)$, and hence the normal quotients of $(Y(r),H_2(r))$ modulo $\ov{N(2,+)}$, and of $(Y(r),H_3(r))$ 
modulo $\ov{N(2)}$, are both isomorphic to $(K_2,Z_2)$. This proves that  lines 3 and 4 of 
Table~\ref{tbl:main2} are valid.

Now let $(\Ga, G)=(X(r),G_3(r))$ and $(\Ga',G')=(Y(r),H_3(r))$. First consider $N=J=\la\mu_1\mu_2\ra$. 
Then the $N$-vertex-orbits in $\Ga$ are the sets $B_i=\{(i+j,j)\mid j\in\mathbb{Z}_{2r}\}$ for $i\in\mathbb{Z}_{2r}$. 
The quotient $(\Ga_N,G/N)$ is the unoriented cycle $C_{2r}$ with edges of both orientations 
between adjacent $N$-orbits $B_i, B_{i+1}$ (for example, $(i,0)\rightarrow (i+1,0)$ and $(i,0)
\leftarrow ((i+1)-1,-1)$). Hence $N$ is the kernel of the $G$-action on the set of $N$-orbits, and 
$(\Ga_N, G/N)=(C_{2r}, D_{4r})$, as in line 5 of Table~\ref{tbl:main2} for `$N$ in $G_3(r)$'.
The same argument with $2r$ replaced by $s$, proves line 5  for `$N$ in $G_{3Z}(s)$' with $N=J$.
Continuing with $N=J$ in $G_3(r)$, since $N$ contains $M(r)$, it follows from  
Lemma~\ref{lem:quotiso} that the normal quotient of $(\Ga',G')$ modulo $\ov{N}$ is also $(C_{2r},D_{4r})$, 
completing the proof of line 5 of Table~\ref{tbl:main2}. 
Moreover, if we replace $N$ by $J(+)=\la \mu_1\mu_2,\mu_1^r\ra$, then
the $N$-vertex-orbits become $B_i\cup B_{i+r}$, for $0\leq i<r$, and the quotients  
$(\Ga_N,G/N)$ and $(\Ga'_{\ov N},G'/\ov N)$ both become $(C_r,D_{2r})$, proving line 7 of Table~\ref{tbl:main2}.

Now consider $M=\la\mu_1\mu_2^{-1}\ra$ in $G=G_3(r)$. The $M$-vertex-orbits in $\Ga$ are the sets 
$D_i=\{(i+j,-j)\mid j\in\mathbb{Z}_{2r}\}$ for $i\in\mathbb{Z}_{2r}$.  
All the out-neighbours of vertices in $B_i$ lie in $B_{i+1}$, and hence the quotient 
is (an oriented) cycle of length $2r$ and the induced group is $Z_{2r}$.
Moreover the element $\tau\in G$ fixes each $D_i$ setwise and the kernel of the 
$G$-action on the set of $M$-orbits is $N:=K=\la M,\tau\ra$. Thus 
$(\Ga_N,G/N)=(C_{2r},Z_{2r})$, and since $N$ contains $M(r)$, it follows 
from  Lemma~\ref{lem:quotiso} that the normal quotient of $(\Ga',G')$ 
modulo $\ov{N}$ is also $(C_{2r},Z_{2r})$,  as in  line 6 of Table~\ref{tbl:main2}.
If we replace $2r$ by $s$ and $(\Ga,G)$ by $(Z(s),G_{3Z}(s))$, the argument above proves line 6
for $N=K$ in $G_{3Z}(s)$, completing the proof of line 6 of Table~\ref{tbl:main2}.
Finally, if we replace $N$ by $K(+)=\la \mu_1\mu_2^{-1},\tau,\mu_1^r\ra$ in $G_3(r)$, then
the $N$-vertex-orbits in $X(r)$ become $D_i\cup D_{i+r}$ for $0\leq i<r$, and the quotients  
$(\Ga_N,G/N)$ and $(\Ga'_{\ov N},G'/\ov N)$ both become $(C_r,Z_{r})$, as asserted in line 8 of Table~\ref{tbl:main2}.
\end{proof}

Theorem~\ref{thm:main} follows from Lemmas~\ref{lem:lines1-4} and \ref{lem:tbl2}.

\section{Identifying the basic pairs $(\Ga,G)$ for Theorem~\ref{thm:mainB}}\label{sec:basic}  

Each of the pairs $(\Ga,G)$ in Theorem~\ref{tbl:mainB} lies in $\OG(4)$, by Lemma~\ref{lem:x2}.
Before competing the proof of Theorem~\ref{thm:mainB} by determining the basic graph-group pairs, we prove a preliminary lemma. The 
\emph{centraliser} of a subgroup $N$ of a group $G$ is the subgroup  $C_G(N) = \{ g\in G\mid gh=hg,\ \forall\ h\in N\}$. For a prime $p$,
$O_p(G)$ is the (unique) largest normal $p$-subgroup of $G$. By Sylow's Theorem, $O_p(G)$ is contained in every Sylow $p$-subgroup of $G$.
Possibly $O_p(G)=1$. 

\begin{lemma}\label{lem:cent} Suppose that $r\geq3$. 
\begin{enumerate}
\item[(a)] Let $N(2)$ be as in $\eqref{na}$, a subgroup of $G(r)$. Then  $C_{G(r)}(N(2))=\la\mu_1,\mu_2\ra$.  
\item[(b)] For $t$ odd, $O_2(G_2(t))=M(t)$, $O_2(G_3(t))=\la\mu_1^t,\mu_2^t\ra$, and $O_2(G_{3Z}(t))=1$. 
\end{enumerate}
\end{lemma} 

\begin{proof}
(a) Let $C:=C_{G(r)}(N(2))$. Recall that $N(2)=\la\mu_1^2,\mu_2^2\ra$ and that $\la\mu_1,\mu_2\ra$ is abelian, so 
$\la\mu_1,\mu_2\ra\leq C$.
Let $K=\la \mu_1,\mu_2,\sigma_1,\sigma_2\ra$. Then each element of $G(r)\setminus K$ interchanges $\la \mu_1\ra$ and 
$\la\mu_2\ra$, and hence does not centralise $N(2)$. Thus $C\leq K$. Similarly, for $i=1,2$, any element of $K$ not 
lying in $\la\mu_1,\mu_2,\sigma_i\ra$ inverts  $\la \mu_{3-i}\ra$, so does not lie in $C$ since $r>2$. It follows that $C=\la\mu_1,\mu_2\ra$.

(b) Let $Q=O_2(G_k(t))$ with $k=2$ or $3$. Since $M(t)=\la\mu_1^t\mu_2^t\ra\cong Z_2$ is normal in $G_k(t)$, we have $M(t)\leq Q$. 
Moreover, since $t$ is odd, the normal subgroups $N(2)$ (or order $t^2$) and $Q$ intersect in the identity subgroup. Hence 
$Q\leq C_{G(t)}(N(2))\cap G_k(t)=L$, say, and by part (a), $L= \la\mu_1,\mu_2\ra\cap G_k(t)$. In fact $Q$ must be contained 
in the unique Sylow $2$-subgroup $P$ of $L$. If $k=2$, then $P=M(t)$ and hence $Q=M(t)$. If $k=3$, then $P=\la\mu_1^t,\mu_2^t\ra$, 
and since this subgroup $P$ is a normal $2$-subgroup of $G_3(t)$, it follows that $Q=P$. Finally consider $Q=O_2(G_{3Z}(t))$.
Since $|G_{3Z}(t)|=2t^2$ with $t$ odd, we have $|Q|\leq2$. Suppose for a contradiction that $|Q|=2$. Then $Q\cap N=1$, where  
$N= \la\mu_1,\mu_2\ra$ (of odd order $t^2$). So $Q$ centralises $N$, but this implies that $G_{3Z}(t)=NQ$ is abelian, which is not 
the case.  Hence $Q=1$.
\end{proof}
 
\begin{lemma}\label{lem:basic}
The `Conditions to be Basic' in Table~$\ref{tbl:mainB}$ are correct, namely, 
\begin{enumerate}
\item[(a)]  $(X(r),G_k(r))$ is basic if and only if $k=1$ and $r$ is an odd prime; 
\item[(b)]  $(Y(r),H_k(r))$ is basic if and only if either $r=2$, or $k=2$ and $r$ is an odd prime; 
\item[(c)]  $(Z(s),G_{3Z}(s))$ is basic if and only if $s$ is an odd prime.  
\end{enumerate}
Moreover the `Basic Type' entries in Table~$\ref{tbl:mainB}$ are also correct.
\end{lemma}

\begin{proof} 
(a)  Suppose first that $(X(r),G_k(r))$ is basic, that is, $(X(r),G_k(r))$ has no proper nondegenerate normal quotients. 
It follows from lines 1 and 2 of Table~\ref{tbl:main} that $k=1$ and $r$ is an odd prime. 

Conversely suppose that $k=1$ and $r$ is an odd prime, and assume, for a contradiction that
$G_1(r)$ has a nontrivial normal subgroup $N$ such that $(X(r)_N,G_k(r)/N)$ is nondegenerate. 
Then by \cite[Theorem 1.1]{janc1}, $N$ is semiregular  on the vertices  of $X(r)$, and this 
quotient has valency 4, so $N$ has at least five vertex-orbits.  Without loss of generality 
we may assume that $N$ is a minimal normal subgroup of $G_1(r)$. Now $N\cap N(2)= 1$ would 
imply that $N\leq C_{G_1(r)}(N(2))$, and hence by Lemma~\ref{lem:cent} that
$N\leq \la\mu_1,\mu_2\ra\cap G_1(r)=N(2)$, which in turn implies $N\leq N\cap N(2)=1$, a 
contradiction. Hence $N\cap N(2)\ne 1$, so by the minimality of $N$, we have $N\leq N(2)$. 
Since $N(2)$ has only four vertex orbits in $X(r)$, $N$ must be a proper subgroup of $N(2)$, 
and since $|N(2)|=r^2$ and $r$ is an odd prime, it follows that $|N|=r$.   Since $N$ is 
normalised by $\tau\in G_1(r)$, it follows that $N\ne\la\mu_i^2\ra$ for $i=1$ or $i=2$, and hence 
that $N=\la\mu_1^2\mu_2^{2i}\ra$ for some $i$ such that $1\leq i<r$. Now $N$ must contain 
$(\mu_1^2\mu_2^{2i})^\tau = \mu_2^2\mu_1^{2i}$. Since the only element of $N$ projecting to 
$\mu_1^{2i}$ is $(\mu_1^2\mu_2^{2i})^i$, we have $\mu_2^2\mu_1^{2i}= (\mu_1^2\mu_2^{2i})^i$, 
and hence $\mu_2^{2(i^2-1)}=1$, so $i^2\equiv 1\pmod{r}$. This implies that $i=1$ or $r-1$. 
However neither $\la\mu_1^2\mu_2^2\ra$ nor $\la\mu_1^2\mu_2^{2(r-1)}\ra$ is normalised by 
$\mu_1\sigma_2\in G_1(r)$. This is a contradiction. Therefore $(X(r),G_k(r))$ is basic 
when $k=1$ and $r$ is an odd prime. Moreover $(X(r),G_1(r))$ is basic of cycle type, 
see line 1 of Table~$\ref{tbl:main2}$. 

(b)  Suppose next that $(\Ga',G')=(Y(r),H_k(r))$ is basic, where $r$ is even if $k=1$. It follows from lines 4, 5
and 6 of Table~\ref{tbl:main} that either $r=2$, or $k=2$ and $r$ is an odd prime.  
Now we prove the converse. 
Suppose that $r=2$, or $k=2$ and $r$ is an odd prime. Let $M$ be a minimal normal 
subgroup of $G'=H_k(r)$ and consider the quotient $\Ga'_M$. If  
$(\Ga'_M,G'/M)$ is nondegenerate (hence of valency $4$), then 
by \cite[Theorem 1.1]{janc1}, $M$ is semiregular  with at least five vertex-orbits on 
$\Ga'$, and $M$ is the kernel of the $G'$-action on the $M$-orbits. 
On the other hand if $(\Ga'_M,G'/M)$ is degenerate, then we replace $M$ by the kernel of 
the $G'$-action on the $M$-orbits. We consider the possibilities for $(\Ga'_M,G'/M)$: in particular 
if none are nondegenerate then $(\Ga',G')$ is basic. %

Suppose first that $r=2$. Then the graph $\Ga'$ has only eight vertices, and hence $M$ has at most four 
vertex-orbits, so all quotients  $(\Ga'_M,G'/M)$ are degenerate. 
Thus if $r=2$, then $(\Ga',G')$ is basic. It is 
basic of cycle type, by lines 1 and 2 of Table~$\ref{tbl:main2}$.

Suppose now that $k=2$ and $r$ is an odd prime. Note that the normal subgroup $\ov{N(2)}$ of $G'$
has just two vertex-orbits, each of size $r^2$, on $\Ga'$. If $M$ contains $\ov{N(2)}$ then by the minimality of $M$,  
$M=\ov{N(2,+)}$ (the kernel of the $G'$-action on the  $\ov{N(2)}$-orbits) and $(\Ga'_M,G'/M)
=(K_2,Z_2)$ as in line 4 of Table~\ref{tbl:main2}. Suppose now that $M\not\supseteq \ov{N(2)}$. 
By Lemma~\ref{lem:quotiso}, since 
$(\Ga',G')$ is isomorphic to the normal quotient of $(\Ga,G)=(X(r),G_k(r))$ modulo the normal subgroup
$N=M(r)$ of $G$ (Table~\ref{tbl:main}, line 2), it follows that $(\Ga'_M,G'/M)$ is isomorphic to
a normal quotient $(\Ga_L,G/L)$ for some normal subgroup $L$ of $G$ such that $N\leq L$ and 
$L$ is the kernel of the $G$-action on the $L$-vertex-orbits on $\Ga$, and $G/L\cong G'/M$.  
Moreover $L/N$ corresponds to $M$ under the isomorphism $G\rightarrow G'$. In particular, 
$L/N$ is a minimal normal subgroup of $G/N$ if $M$ is a minimal normal subgroup of $G'$, 
and since  $M\not\supseteq \ov{N(2)}$ we have $L\not\supseteq N(2)$. 

If $|L|$ is not divisible by (the odd prime) $r$, then $L$ is a normal 2-subgroup of $G=G_2(r)$ 
properly containing $N=M(r)$. Hence $L\leq O_2(G)$ and $O_2(G)\ne M(r)$, contradicting Lemma~\ref{lem:cent}(b).  
Thus $|L|$ is divisible by $r$. 
Then $L':=L\cap N(2)\ne 1$, and $L'\ne N(2)$ since $L\not\supseteq N(2)$. Since $|N(2)|=r^2$,
it follows that $|L'|=r$,  
and, being the intersection of two normal subgroups, $L'$ is normal in $G=G_2(r)$. 
We argue as in the proof of part (a): $L'\ne\la\mu_i^2\ra$ for $i=1$ or $i=2$ 
since $L'$ is normalised by $\tau\mu_1$. So $L'=\la\mu_1^2\mu_2^{2i}\ra$, for some 
$i$ such that $1\leq i<r$. However $\sigma_2\in G_2(r)$ and $\sigma_2$ does not 
normalise $L'$, contradiction. Thus there is no proper 
normal quotient $(\Ga'_M,G'/M)$ with $M\not\supseteq \ov{N(2)}$.
Hence $(\Ga',G')$ is basic and its only proper normal quotient is $(K_2,Z_2)$.
This implies that $(\Ga',G')$ is basic of biquasiprimitive type, see Table~\ref{tbl:basictypes},
and completes the proof of part (b).

\smallskip
(c)  Suppose now that $(\Ga,G)=(Z(s),G_{3Z}(s))$, where $s$ is odd. If $(\Ga,G)$ is basic
then it follows from line 7 of Table~\ref{tbl:main} that $s$ is an odd prime. Suppose conversely 
that $s$ is an odd prime.  Let $M$ be a nontrivial normal subgroup of $G$ which is equal to the 
kernel of the $G$-action on the $M$-orbits in $\Ga$, and consider $(\Ga_M,G/M)$. If $M$ contains 
$N=N(1)=\la\mu_1,\mu_2\ra$ (of order $s^2$), then since $N$ is vertex-transitive on $Z(s)$, we have
$M=G$ and  $(\Ga_M,G/M)= (K_1,1)$. So assume 
that $N\not\subseteq M$. If $M\cap N=1$, then $M$ is a normal subgroup of
$G$ of order dividing $|G|/|N|=2$, and so $|M|=2$ and $M\leq O_2(G)$, contradicting Lemma~\ref{lem:cent}.
Thus $M\cap N$ must have order $s$, and must be normal in $G$.  Since $M\cap N$ is normalised by 
$\tau\in G$, it follows that $M\cap N=\la\mu_1\mu_2^{\pm 1}\ra$, that is, $M\cap N$ is either $J$
 or the derived subgroup $K'$ of $K$, as defined in \eqref{ld}. If $M\cap N = K'$, then the orbits of $M\cap N$ and $K$ are the same.
Thus $\Ga_M$ is a quotient of $\Ga_{M\cap N}$ which, in either case, is a cycle of length $s$, by lines 5 and 
6 of Table~\ref{tbl:main2}. Thus $(\Ga,G)$ is basic of cycle type, completing the proof of part (c).
\end{proof}

Finally we observe that Theorem~\ref{thm:mainB} follows from Lemma~\ref{lem:x2} (for membership in $\OG(4)$),
and  Lemma~\ref{lem:basic}.

\end{document}